\documentclass[12pt]{amsart}
\usepackage{amsmath,amsthm,latexsym,amscd,amsbsy,amssymb,url}
\usepackage[all]{xypic}
 \relax

\chardef\bslash=`\\ % p. 424, TeXbook

\makeatletter
\def\verbatim{\interlinepenalty\@M \@verbatim
  \leftskip\@totalleftmargin\advance\leftskip2pc
  \frenchspacing\@vobeyspaces \@xverbatim}
\makeatother
\newtheorem*{theorem*}{Theorem}

\newtheorem{thm}{Theorem}[section]
\newtheorem{cor}[thm]{Corollary}
\newtheorem{lem}[thm]{Lemma}
\newtheorem{pro}[thm]{Proposition}

\newtheorem*{A}{Theorem 4.1}
\newtheorem*{B}{Theorem 4.4}
\newtheorem*{C}{Theorem 6.7}

\theoremstyle{definition}
\newtheorem{defin}{Definition}[section]

\theoremstyle{remark}

\numberwithin{equation}{section}

\begin{document}

%%%%%%% Begin Topmatter %%%%%%%%%%

\title
{Extension properties of Stone-\v{C}ech coronas and proper absolute extensors}
\author{A. Chigogidze}
\address{Department of Mathematics,
College of Staten Island, CUNY,
2800 Victory Blvd, Staten Island, NY, 10314, USA}
\email{alex.chigogidze@csi.cuny.edu}
\keywords{Stone-\v{C}ech corona, $Z_{\tau}$-set, finite complex, absolute extensor}
\subjclass{Primary: 54C20, 57N20; Secondary: 54D35}

%%%%%%% End topmatter %%%%%%%%%

\begin{abstract}{We characterize, in terms of $X$, extensional dimension of the Stone-\v{C}ech corona $\beta X \setminus X$ of locally compact and Lindel\"{o}f space $X$. The non-Lindel\"{o}f case case is also settled in terms of extending proper maps with values in $I^{\tau}\setminus L$, where $L$ is a finite complex. Further, for a finite complex $L$, an uncountable cardinal $\tau$ and a $Z_{\tau}$-set $X$ in the Tychonov cube $I^{\tau}$ we find necessary and sufficient condition, in terms of $I^{\tau}\setminus X$, for $X$ to be in the class $\operatorname{AE}([L])$. We also introduce a concept of a proper absolute extensor and characterize the product $[0,1)\times I^{\tau}$ as the only locally compact and Lindel\"{o}f proper absolute extensor of weight $\tau > \omega$ which has the same pseudocharacter at each point.}
\end{abstract}

\maketitle
\markboth{A.~Chigogidze}
{Stone-\v{C}ech coronas}

\section{Introduction}
We study extension properties of Stone-\v{C}ech coronas of locally compact spaces focusing on the following two problems:

\begin{itemize}
  \item[(A)]
When -- in terms of $X$ -- are maps, defined on closed subsets of $\beta X \setminus X$, into a finite complex $L$ extendible to the whole $\beta X \setminus X$? 
  \item[(B)] 
 When -- in terms of $Y$ -- are maps, defined on closed subsets of nice spaces, into $\beta Y \setminus Y$ extendible to the whole domain? 
\end{itemize}

When every map $f \colon A \to Y$, defined on a closed subset $A$ of $X$, has an extension $\bar{f} \colon X\to Y$ we say that $Y$ is an absolute extensor of $X$ and write $Y \in \operatorname{AE}(X)$. Assuming that both $f$ and $\bar{f}$ in this definition are proper  we obtain notion of proper absolute extensor (for details see Definition \ref{D:main}). Notation for the latter is $Y \in \operatorname{AE}_{p}(X)$. It turns out (Corollaries \ref{C:first}, \ref{C:dim}) that for a locally compact and Lindel\"{o}f (e.g. separable and metrizable) space $X$ and a finite complex $L$, $L \in \operatorname{AE}(\beta X \setminus X)$ precisely when $\operatorname{Cone}(L)\setminus L \in \operatorname{AE}_{p}(X)$ (here $L$ is identified with the base $L \times \{ 0\}$ of the cone $\operatorname{Cone}(L)$). For $L = S^{n}$, we obtain the following observation: $\dim (\beta X \setminus X) = \dim_{p}X -1$, where $\dim_{p}X \leq n$ is just a notation for $R^{n}\in \operatorname{AE}_{p}(X)$. We should point out that the problem of describing dimensions (covering, inductive) of the Stone-\v{C}ech (or Hewitt) coronas, using completely different approaches, has been considered by several authors (see, for example,  \cite{aarts}, \cite{an}, \cite{an1}, \cite{smirnov1}, \cite{smirnov2}, \cite{chif}). 

However, non Lindel\"{o}f spaces do not admit proper maps into $R^{n}$ or into any space of the form $\operatorname{Cone}(L) \setminus L$, where $L$ is a finite complex, and above observations need to be adjusted in order to remain valid in general case. We start by noting that since $L$ (i.e. $L \times \{0 \} \subset \operatorname{Cone}(L)$) is a $Z$-set in $\operatorname{Cone}(L)$ it follows from the Chapman's Complement Theorem that no matter how is $L$ $Z$-embedded into the Hilbert cube the complement $I^{\omega}\setminus L$ is homeomorphic to $I^{\omega} \times \operatorname{Cone}(L) \setminus I^{\omega} \times L = I^{\omega} \times (\operatorname{Cone}(L) \setminus L)$. Since $\operatorname{Cone}(L) \setminus L  \in\operatorname{AE}_{p}(X)$ if and only if $I^{\omega} \times (\operatorname{Cone}(L) \setminus L) \in\operatorname{AE}_{p}(X)$, observation made above can be reformulated as follows: $L \in \operatorname{AE}(\beta X \setminus X)$ if and only if $I^{\omega} \setminus L \in \operatorname{AE}_{p}(X)$. While the testing space $I^{\omega}\setminus L$ is still Lindel\"{o}f and hence is not suitable for general situation, it does allow us to find it's non-metrizable counterpart, which turns out is the complement $I^{\tau} \setminus L$. Choice of an embedding $L \hookrightarrow I^{\tau}$, when $\tau > \omega$, is irrevelevant since any metric compactum is a $Z_{\tau}$-set in $I^{\tau}$ as long as $\tau > \omega$ \cite[Corollary 8.5.7]{chibook}. With this in mind we settle problem (A) by proving the following statement.

\begin{A} 
Let $X$ be a locally compact space which can be covered by at most $\tau$ compact subsets and each regular closed subset of which is $C^{\ast}$-embedded. Let also $L$ be a compact $ANR$-space embedded into the cube $I^{\tau}$ as a $Z_{\tau}$-set. Then the following conditions are equivalent:
\begin{itemize}
  \item[(a)]
$L \in \operatorname{AE}(\beta X \setminus X)$; 
  \item[(b)]
$I^{\tau} \setminus L \in \operatorname{AE}_{p}(X)$. 
\end{itemize}  
\end{A}

Problem (B), in some cases, can also be settled in a similar manner. Specifically, we consider spaces of the form $Y = I^{\tau}\setminus X$, where $X$ is a $Z_{\tau}$-set in $I^{\tau}$. For $\tau > \omega$, $I^{\tau}$ is indeed the Stone-\v{C}ech compactification of $Y$ (Lemma \ref{L:SC}). In this situation problem (B) becomes a part of a general problem of recovering properties of $X$ in terms of its complement $I^{\tau} \setminus X$. This leads us to considerations very similar to the study carried out in \cite{cdkm} for $\tau = \omega$. However, there is a major difference between the metrizable $(\tau = \omega$) and non-metrizable ($\tau > \omega$) cases. Roots of this difference, one could argue, lie in the fact that the topological type of the complement $I^{\omega}\setminus X$ of a $Z$-set in the Hilbert cube, while determining $X$'s shape, does not uniquely determine topological type of $X$. But if $\tau > \omega$, topological type of any $Z_{\tau}$-set $X$ in $I^{\tau}$ is completely datermined by its complement. This is apparently why we need to exploit metric-uniform invariants in the metrizable case (see \cite{cdkm}) and why we could remain in the topological category if $\tau > \omega$. Going back to problem (B), it turns out that -- as in problem (A) -- the complements $I^{\tau}\setminus L$ of finite complexes still play a critical role. In order to formulate our second result let us recall that the extension class $[L]$ of a complex is a collection of all extensionally equivalent complexes ($K$ is equivalent to $L$ if $K \in \operatorname{AE}(X)$ holds if and only if $L \in \operatorname{AE}(X)$ for any $X$). We say that $X \in \operatorname{AE}([L])$ if $X \in \operatorname{AE}(Y)$ whenever $L \in \operatorname{AE}(Y)$. Similarly, we can define a proper extensional class $\operatorname{AE}_{p}^{\tau}([I^{\tau}\setminus L])$ by agreeing that $Y \in \operatorname{AE}_{p}^{\tau}([I^{\tau}\setminus L])$, where $Y$ is a locally compact space of weight $\leq \tau$, if $Y \in \operatorname{AE}_{p}(M)$ for any locally compact space $M$ of weight $\leq \tau$ with $I^{\tau}\setminus L \in  \operatorname{AE}_{p}(M)$. We prove the following statement.

\begin{B}
Let $\tau > \omega$, $L$ be a compact $ANR$-space embedded into $I^{\tau}$ as a $Z_{\tau}$-set  and $X$ be a $Z_{\tau}$-set in $I^{\tau}$. Then the following conditions are equivalent:
\begin{itemize}
  \item[(i)]
$X \in \operatorname{AE}([L])$; 
  \item[(ii)] 
$I^{\tau}\setminus X \in \operatorname{AE}_{p}^{\tau}([I^{\tau}\setminus L])$.  
\end{itemize}
\end{B}

These considerations lead to the concept of a proper absolute extensor which we study in Section \ref{S:PAE} (see \cite{nov}, \cite{Michael} for related results). Note that $R^{n}$ is not a proper absolute extensor for any $n$ (while it is, of course, an absolute extensor). To see this in case $n=1$ note that the proper map $f \colon N \to R$, defined by 
\[ f(n) =
\begin{cases}
n\; , \; n\; \text{ is odd};\\
-n\; , n\; \text{is even}
\end{cases}
\]
\noindent does not have a proper extension $\bar{f} \colon R \to R$. On the other hand, $R^{n}_{+} = \{ (x_{i})_{i = 1}^{n} \in R^{n} \colon x_{n} \geq 0 \}$ is a proper absolute extensor for each $n$. Explanation of this fact (see Lemma \ref{L:manifold}) is that $R^{n}_{+}$ has a compactification (namely, $I^{n}$) which is an absolute extensor and that the corresponding corona ($I^{n-1}$) is also an absolute extensor, sitting in $I^{n}$ as a $Z$-set. We show that the only proper absolute extensor of  countable weight satisfying $DD^{n}P$ for each $n$ is the product $[0,1)\times I^{\omega}$ (Proposition \ref{P:1}). In the non-metrizable case we have the following statement.

\begin{C}\label{T:C}
A proper absolute extensor of weight $\tau > \omega$ is homeomorphic to the product $[0,1)\times I^{\tau}$ if and only if it has the same pseudocharacter at each point.  
\end{C}

The paper is organized as follows. In Section \ref{S:TC}, based on modified versions of \v{S}\v{c}epin's Spectral Theorem, we obtain characterization of $Z_{\tau}$-sets in the Tychonov cube $I^{\tau}$ and prove the mapping replacement results (Propositions \ref{L:main} and \ref{P:adjusted}). These are then used in Section \ref{S:corona} to prove Theorems \ref{T:dim} and \ref{T:main}. In Section \ref{S:category} we extend results, obtained in \cite{cdkm} for the Hilbert cube, to the Tychonov cube. Namely, we describe topological and homotopy categories of $Z_{\tau}$-sets in $I^{\tau}$ in terms of certain naturally defined categories of their complements. Considerations here involve certain concepts of coarse geometry which are still relevant in the topological setting. In the final Section \ref{S:PAE} we investigate concept of proper absolute extensor and prove Theorem \ref{T:1}.

%%%%%%%%%%%%%%%%%%%%%%%%%%%%%%%%%%%%%%%%%%%%%%%%%%%%%%%%
%%%%%%%%%%%%%%%%%%%%%%%%%%%%%%%%%%%%%%%%%%%%%%%%%%%%

\section{Preliminaries}\label{S:pre}

Unless noted otherwise below we consider only locally compact Tychonov spaces and continuous maps.  A map $f \colon X \to Y$ is proper if $f^{-1}(C)$ is compact for any compact $C \subset Y$. Note that the class of proper maps between locally compact spaces coincides with the class of perfect maps (a map is perfect if it is closed and has compact point inverses). A set $F \subset X$ is $z$-embedded in $X$ if for every functionally closed (in $F$) set $Z \subset F$ there exists a functionally closed set $\widetilde{Z}$ in $X$ such that $Z = F \cap \widetilde{Z}$. A set $F \subset X$ is $C^{\ast}$-embedded if every bounded real-valued continuous function, defined on $F$, has a bounded and continuous extension, defined on $X$.

\begin{lem}\label{L:SC}
Let $\tau > \omega$ and $X$ be an open and $G_{\delta}$-dense subset of the Tychonov cube $I^{\tau}$. Then
\begin{itemize}
\item[(i)]
$X$ is pseudocompact and $\beta X = I^{\tau}$;
\item[(ii)]
If $F$ is a functionally closed subset of $I^{\tau}$, then $F\cap X$ is $C^{\ast}$-embedded in $X$;
\item[(iii)]
If $G$ is an open subset of $X$, then $\operatorname{cl}_{X}G$ is $C^{\ast}$-embedded in $X$.
\end{itemize}
\end{lem}
\begin{proof}
(i). Since $X$ is dense in $I^{\tau}$ it follows (see \cite[Corollary 6.4.7]{chibook}) that $X$ is $z$-embedded in $I^{\tau}$. Since, by assumption, $I^{\tau}\setminus X$ does not contain functionally closed subsets of $I^{\tau}$, we conclude (\cite[Proposition 1.1.22]{chibook}) that $I^{\tau}$ is the Hewitt realcompactification of $X$. Compactness of $I^{\tau}$ implies that $I^{\tau}$ is actually the Stone-\v{C}ech compactification of $X$ and $X$ is pseudocompact.

(ii). Since $X$ is $G_{\delta}$-dense, it follows that $F \cap X \neq \emptyset$. By (i) and \cite[8D.1]{gj}, $F = \operatorname{cl}_{I^{\tau}}(F \cap X)$. Since the cube $I^{\tau}$ is an $\operatorname{AE}(0)$-space, it follows from \cite[Propositions 6.1.8, 6.4.9]{chibook} that $F$ itself is an $\operatorname{AE}(0)$-space. Consequenlty, by \cite[Proposition 1.1.21(ii)]{chibook}, $F \cap X$ is $z$-embedded in $F$. Since $F\cap X$ is $G_{\delta}$ -dense in $F$, it follows from \cite[Proposition 1.1.22]{chibook} that $F$ is the Stone-\v{C}ech compactification of $F\cap X$. Then $F \cap X$ is $C^{\ast}$-embedded in $X$. 

(iii). Clearly, $\operatorname{cl}_{X}G = X \cap \operatorname{cl}_{I^{\tau}}G$. Since the latter set is functionally closed in $I^{\tau}$, the needed conclusion follows from (ii).
\end{proof}

Extension theory -- a generalization of the classical dimension theory -- as developed by A.~Dranishnikov, as well as certain facts from infinite-dimensional topology  (see \cite{chihand} for unified treatment of both) are used without specific references.

%%%%%%%%%%%%%%%%%%%%%%%%%
%%%%%%%%%%%%%%%%%%%%%%%%%%

\section{$Z_{\tau}$-sets in the Tychonov Cube}\label{S:TC}
In this section we study certain properties of $Z_{\tau}$-sets in the Tychonov cube $I^{\tau}$ introduced in \cite{chibook}.

\subsection{Spectral Theorem -- revisited}\label{SS:st}

We begin by establishing needed versions of \v{S}\v{c}epin's Spectral Theorem \cite[Theorem 1.3.4]{chibook}.

\begin{pro}\label{P:spectralwcontrol}
Let $\tau \geq \omega$, $|T| > \tau$, $T_{0} \subseteq T$, $|T_{0}| < |T|$ and $g \colon \prod\{ X_{t} \colon t \in T\} \to \prod\{ X_{t} \colon t \in T\}$ be a map of the product of compact metrizable spaces such that $\pi_{T_{0}}\circ g = \pi_{T_{0}}$. Then the set

\[
 {\mathcal M}_{(g,T_{0})}=\biggl\{ R \subseteq \exp_{\tau}(T\setminus T_{0}) \colon \exists g_{T_{0}\cup R} \colon \prod\{ X_{t} \colon t \in T_{0} \cup R\} \to \prod\{ X_{t} \colon t \in T_{0} \cup R\}\biggr.\]
\[ \; \text{with}\;
 \biggl. \pi_{T_{0}\cup R}\circ g = g_{T_{0}\cup R}\circ \pi_{T_{0}\cup R}\biggr\}     
\]

\noindent is cofinal and $\tau$-closed in $\exp_{\tau}(T\setminus T_{0})$.
\end{pro}
\begin{proof}
By \cite[Theorem 1.3.4]{chibook}, the set

\[ 
{\mathcal M}_{g}=\Biggl\{ R \in \exp_{\tau}T \colon \exists  g_{R} \colon \prod\{ X_{t} \colon t \in R\} \to \prod\{ X_{t} \colon t \in R\}\;\biggr.\]
\[\; \text{with}\; \biggl. \pi_{R}\circ g = g_{R}\circ \pi_{R}\Biggr\}
\]

\noindent is cofinal and $\tau$-closed in $\exp_{\tau}T$.

Let $S \in \exp_{\tau}(T\setminus T_{0})$ and choose $\tilde{R} \in {\mathcal M}_{g}$ such that $S \subseteq \tilde{R}$. The corresponding $g_{\tilde{R}}$ does not change the $X_{t}$-coordinate for $t \in \tilde{R}\cap T_{0}$ (since $\pi_{T_{0}}\circ g = \pi_{T_{0}}$). Consequently the diagonal product

\begin{multline*}
 g_{T_{0}\cup \tilde{R}} = \pi_{T_{0}}^{T_{0}\cup \tilde{R}}\triangle \pi_{\tilde{R}\setminus T_{0}}^{\tilde{R}}g_{\tilde{R}} \pi_{\tilde{R}}^{T_{0}\cup \tilde{R}} \colon \prod\{ X_{t} \colon t \in T_{0}\cup \tilde{R}\} \to\\ \prod\{ X_{t}\colon t \in T_{0}\}\times \prod\{ X_{t} \colon t \in \tilde{R}\setminus T_{0}\} \end{multline*}

\noindent is well defined. Set $R = \tilde{R}\setminus T_{0}$. Obviously, $S \subseteq R$ and $R \in {\mathcal M}_{(g,T_{0})}$ which proves that ${\mathcal M}_{(g,T_{0})}$ is cofinal in $\exp_{\tau}(T\setminus T_{0})$. The $\tau$-completeness of ${\mathcal M}_{(g,T_{0})}$ in $\exp_{\tau}(T\setminus T_{0})$ is obvious.
\end{proof}

%%%%%%%%%%%%%%%%%%%%%%%%%%%%%%%%%%%%%

\begin{cor}\label{C:spectralwcontrol}
Let $\tau \geq \omega$, $|T| > \tau$, $T_{0} \subseteq T$, $|T_{0}| < |T|$ and $f \colon X\to Y$ be a map between closed subspaces of the Tychonov cube $I^{T}$. If $\pi_{T_{0}}\circ f = \pi_{T_{0}}|X$. Then the set

\[
 {\mathcal M}_{(f,T_{0})}=\biggl\{ R \subseteq \exp_{\tau}(T\setminus T_{0}) \colon \exists f_{T_{0}\cup R} \colon \pi_{T_{0}\cup R}(X) \to \pi_{T_{0}\cup R}(Y )\; \text{with}\biggr.\]
\[ 
 \biggl. \pi_{T_{0}\cup R}\circ f = f_{T_{0}\cup R}\circ \pi_{T_{0}\cup R}|X\biggr\}     
\]

\noindent is cofinal and $\tau$-closed in $\exp_{\tau}(T\setminus T_{0})$.
\end{cor}
\begin{proof}
Let $g \colon I^{T}  \to I^{T}$ be a map such that $g|X = f$ and $\pi_{T_{0}}\circ g = \pi_{T_{0}}$. By Proposition \ref{P:spectralwcontrol}, the set

\[
 {\mathcal M}_{(g,T_{0})}=\biggl\{ R \subseteq \exp_{\tau}(T\setminus T_{0}) \colon \exists g_{T_{0}\cup R} \colon \prod\{ X_{t} \colon t \in T_{0} \cup R\} \to \prod\{ X_{t} \colon t \in T_{0} \cup R\}\biggr.\]
\[ \; \text{with}\;
 \biggl. \pi_{T_{0}\cup R}\circ g = g_{T_{0}\cup R}\circ \pi_{T_{0}\cup R}\biggr\}     
\]

\noindent is cofinal and $\tau$-closed in $\exp_{\tau}(T\setminus T_{0})$.

For each $R \in {\mathcal M}_{(g,T_{0})}$ let $f_{T_{0}\cup R} = g_{T_{0}\cup R}|g_{T_{0}\cup R}(X)$.
\end{proof}

\begin{pro}\label{P:spectralwcontrolh}
If, in Proposition \ref{P:spectralwcontrol}, the map $g$ is a homeomorphism, then the set 

\[ {\mathcal H}_{(g, T_{0})} = \{ R \in {\mathcal M}_{(g,T_{0})} \colon g_{T_{0} \cup R}\; \text{is a homeomorphism\;}\}
\]

\noindent is cofinal and $\tau$-closed in $\exp_{\tau}(T\setminus T_{0})$.
\end{pro}
\begin{proof}
By Proposition \ref{P:spectralwcontrol} applied both to $g$ and $g^{-1}$, the sets ${\mathcal M}_{(g,T_{0})}$ and ${\mathcal M}_{(g^{-1},T_{0})}$ are cofinal and $\tau$-closed in $\exp_{\tau}(T\setminus T_{0})$. By \cite[Proposition 1.1.27]{chibook}, ${\mathcal M}_{(g,T_{0})} \cap {\mathcal M}_{(g^{-1},T_{0})}$ is still cofinal and $\tau$-closed. It only remains to note that for each $R$ from this intersection the map $g_{T_{0}\cup R}$ is a homeomorphism.
\end{proof} 

\begin{cor}\label{C:spectralwcontrolh}
If, in Corollary \ref{C:spectralwcontrol}, the map $f$ is a homeomorphism, then the set 

\[ {\mathcal H}_{(f, T_{0})} = \{ R \in {\mathcal M}_{(f,T_{0})} \colon f_{T_{0} \cup R}\; \text{is a homeomorphism\;}\}
\]

\noindent is cofinal and $\tau$-closed in $\exp_{\tau}(T\setminus T_{0})$.
\end{cor}

%%%%%%%%%%%%%%%%%%%%%%%%%%%%%%%%%%%%%%%%%%%%%%%%%%%%%%%%%%%%%%%%%%%%%%%%%%%%%%%%
%%%%%%%%%%%%%%%%%%%%%%%%%%%%%%%%%%%%%%%%%%%%%%%%%%%%%%%%%%%%%%%%%%%%%%%%%%%%%%%

\subsection{Properties of $Z_{\tau}$-sets in the Tychonov cube $I^{\tau}$}\label{SS:zsets}

By $\operatorname{cov}(X)$ we denote the collection of all countable functionally open covers of a space $X$. We introduce the following notation

\[ B(f,\{ {\mathcal U}_{t} \colon t \in T\} ) = \{ g \in C(X,Y) \colon g\; \text{is}\; {\mathcal U}_{t}\text{-close to}\; f \;\text{for each}\; t \in T\} , \]
\medskip

Let $\tau$ be an infinite cardinal. If $X$ and $Y$ are  Tychonov spaces then $C_{\tau}(X,Y)$ denotes the space of all continuous maps $X \to Y$ with the topology defined as follows (\cite{chifund}, \cite[p.273]{chibook}):  a set $G \subseteq C_{\tau}(X,Y)$ is open if for each $h \in G$ there exists a collection $\{ {\mathcal U}_{t} \colon t \in T\} \subseteq \operatorname{cov}(Y)$, with $|T| < \tau$, such that

\[ h \in B(f,\{ {\mathcal U}_{t} \colon t \in T\} ) \subseteq G .\]

Obviously if $\tau = \omega$, then the above topology coincides with the limitation topology (see \cite{torun}). For $\tau > \omega$, this topology is conveniently described in the following statement \cite[Lemma 3.1]{chiext}.

\begin{lem}\label{L:description}
Let $\tau > \omega$ and $X$ be a $z$-embedded subspace of a product $\prod\{ X_{t} \colon t \in T\}$ of separable metrizable spaces. If $|T| = \tau$, then basic neighborhoods of a map $f \colon Y \to X$ in $C_{\tau}(Y,X)$ are of the form $B(f, S) = \{ g \in C_{\tau}(Y,X) \colon \pi_{S}\circ g = \pi_{S}\circ f\}$, $S \subset T$, $|S| < \tau$, where $\pi_{S} \colon \prod\{ X_{t} \colon t \in T\} \to \prod\{ X_{t} \colon t \in S\}$ denotes the projection onto the corresponding subproduct.
\end{lem}
Now we are ready to define $Z_{\tau}$-sets \cite[Definition 8.5.1]{chibook}.

\begin{defin}\label{D:zset}
Let $\tau \geq \omega$. A closed subset $A \subset X$ is a $Z_{\tau}$-set in $X$ if the set $\{ f \in C_{\tau}(X,X) \colon f(X) \cap A =\emptyset\}$ is dense in the space $C_{\tau}(X,X)$. 
\end{defin}
Clearly $Z_{\omega}$-sets are same as standard $Z$-sets. We also need the following concept.

\begin{defin}\label{D:fiberedzset}
Let $\pi \colon X \to Y$ be a map. A closed subset $A \subset X$ is a fibered $Z$-set in $X$ if the set $\{ f \in C^{\pi}_{\tau}(X,X) \colon f(X) \cap A=\emptyset\}$ is dense in the space $C^{\pi}_{\tau}(X,X)= \{ f \in C_{\tau}(X,X) \colon \pi \circ f = \pi\}$.
\end{defin}

\begin{lem}\label{L:zset1}
Let $\tau > \omega$ and $|T| = \tau$. For a closed set $M \subset I^{T}$ the following conditions are equivalent:
\begin{enumerate}
\item
$M$ is a $Z_{\tau}$-set in $I^{T}$,
\item
For each $T_{0} \subset T$, with $|T_{0}| <\tau$, the set

\begin{multline*}
 {\mathcal Z}_{(M,T_{0})}=\biggl\{ S \subset \exp_{\omega}(T\setminus T_{0}) \colon \pi_{T_{0} \cup S}(M) \;\text{is a fibered}\; Z\text{-set in}\; I^{T_{0}\cup S}\; \text{with}\\  
   \text{respect to}\; \pi_{T_{0}}^{T_{0}\cup S}\biggr\}  
\end{multline*}

\noindent is cofinal and $\omega$-closed in $\exp_{\omega}(T\setminus T_{0})$.
\end{enumerate}
\end{lem}
\begin{proof}
(1) $\Longrightarrow$ (2). Let $S_{0} =R \in \exp_{\omega}(T\setminus T_{0})$. Since $M$ is a $Z_{\tau}$-set in $I^{T}$ and $|T_{0} \cup S_{0}|<\tau$, there exists, by Lemma \ref{L:description}, a map $f_{1} \colon I^{T} \to I^{T}$ such that $\pi_{T_{0}\cup S_{0}}\circ f_{1} = \pi_{T_{0}\cup S_{0}}$ and  $f_{1}(I^{T}) \cap M = \emptyset$. By Proposition \ref{P:spectralwcontrol}, there exist a countable subset $S_{1} \subseteq T\setminus T_{0}$, with $S_{0} \subseteq S_{1}$, and a map $g_{1} \colon I^{T_{0} \cup S_{1}} \to I^{T_{0} \cup S_{1}}$ such that $g_{1}(I^{T_{0} \cup S_{1}}) \cap \pi_{T_{0}\cup S_{1}}(M) = \emptyset$ and $\pi_{T_{0} \cup S_{1}}\circ f_{1} = g_{1}\circ \pi_{T_{0} \cup S_{1}}$ . Note that $\pi_{T_{0}}^{T_{0}\cup S_{1}}\circ g_{1} = \pi_{T_{0}}^{T_{0}\cup S_{1}}$. 

Continuing this process we construct an increasing sequence $\{ S_{n} \colon n \in \omega\}$ of countable subsets of $T\setminus T_{0}$ and maps $g_{n} \colon I^{T_{0}\cup S_{n}} \to I^{T_{0}\cup S_{n}}$ so that $\pi_{T_{0}\cup S_{n}}^{T_{0}\cup S_{n+1}}\circ g_{n+1} = \pi_{T_{0}\cup S_{n}}^{T_{0}\cup S_{n+1}}$ and $g_{n}(I^{T_{0}\cup S_{n}}) \cap \pi_{T_{0}\cup S_{n}}(M) = \emptyset$ for each $n \geq 1$. Let $S = \cup \{ S_{n}\colon n \in \omega\}$. We leave to the reader verification of the fact that the set $\pi_{T_{0}\cup S}(M)$ is a fibered $Z$-set in the cube $I^{T_{0}\cup S}$ with respect to the projection $\pi_{T_{0}}^{T_{0}\cup S}$ (consult with the proof of \cite[Proposition 2.3]{chiext}) . This proves the cofinality of the set  ${\mathcal Z}_{(M,T_{0})}$. The $\omega$-completeness of this set is obvious.
 
(2) $\Longrightarrow$ (1). According to Lemma \ref{L:description} it suffices to find, for any $T_{0} \subset T$ with $|T_{0}|<\tau$, a map $f \colon I^{T} \to I^{T}$ such that $\pi_{T_{0}}\circ f = \pi_{T_{0}}$ and $f(I^{T}) \cap M = \emptyset$. By (2), there exist a countable subset $S \subset T \setminus T_{0}$ and a map $g \colon I^{T_{0}\cup S} \to I^{T_{0}\cup S}$ such that $\pi_{T_{0}}^{T_{0}\cup S}\circ g = \pi_{T_{0}}^{T_{0}\cup S}$ and $g(I^{T_{0}\cup S}) \cap \pi_{T_{0}\cup S}(M)$. Let $j \colon I^{T_{0}\cup S} \to I^{T}$ be a section of the projection $\pi_{T_{0}\cup S} \colon I^{T} \to I^{T_{0}\cup S}$. It only remains to note that the map $f = j\circ g \circ \pi_{T_{0}\cup S}$ has required properties.
\end{proof}

\begin{pro}\label{P:zset2}
Let $\tau > \omega$ and $|T| = \tau$. For a closed set $M \subset I^{T}$ the following conditions are equivalent:
\begin{enumerate}
  \item
$M$ is a $Z_{\tau}$-set in $I^{T}$;
  \item 
If $F\subset M$ is a closed subset, then $\psi(F,I^{T}) =\tau$;
  \item 
$T$ can be represented as the increasing union $T = \cup\{ T_{\alpha} \colon \alpha < \tau \}$ of its subsets so that
\begin{itemize}
  \item[(a)]
$|T_{0}| = \omega$; 
  \item[(b)] 
$|T_{\alpha +1} \setminus T_{\alpha}| = \omega$;
  \item[(c)]
$T_{\alpha} = \cup\{ T_{\beta} \colon \beta < \alpha \}$ for each limit ordinal $\alpha < \tau$; 
  \item[(d)]   
$\pi_{T_{\alpha +1}}(Z)$ is a fibered $Z$-set in $I^{T_{\alpha +1}}$ with respect to the projection $\pi_{T_{\alpha}}^{T_{\alpha +1}} \colon I^{T_{\alpha +1}} \to I^{T_{\alpha}}$. 
\end{itemize}
\end{enumerate}
\end{pro}
\begin{proof}
Equivalence (1) and (2) is proved in \cite[Proposition 8.5.5]{chibook}. Implication (2) $\Longrightarrow$ (3) follows from Lemma \ref{L:zset1}. Finally, in order to prove implication (3) $\Longrightarrow$ (1), let $A\subset T$ with $|A| < \tau$. Then, by (b), there exists $\alpha < \tau$ such that $A \subset T_{\alpha}$. By (d), there exists a map $f \colon I^{T_{\alpha +1}} \to I^{T_{\alpha +1}}$ such that $\pi_{T_{\alpha}}^{T_{\alpha +1}}\circ f = \pi_{T_{\alpha}}^{T_{\alpha +1}}$ and $\operatorname{Im}(f) \cap \pi_{T_{\alpha +1}}(Z) = \emptyset$. Let $g \colon I^{T} \to I^{T}$ is defined as $g = i_{T_{\alpha +1}}\circ f \circ \pi_{T_{\alpha +1}}$, where $i_{T_{\alpha +1}}$ is a section of the projection $\pi_{T_{\alpha +1}}$. It is clear that $\pi_{A}\circ g = \pi_{A}$ and $\operatorname{Im}(g) \cap Z = \emptyset$.
\end{proof}

\begin{cor}\label{C:embedding}
Let $\tau \geq \omega$ and $X$ be a compact space of weight $\leq \tau$. Then there exists a $Z_{\tau}$-embedding of $X$ into the cube $I^{\tau}$.
\end{cor}
\begin{proof}
For $\tau = \omega$ the statement is known. Let $|T| = \tau > \omega$ and represent the cube $I^{\tau}$ as a product $\prod\{ Y_{t} \colon t \in T\}$, where each $Y_{t}$ is a copy of $I^{\omega}$. Embed $X$ into a product $\prod\{ X_{t} \colon t \in T\}$ of compact metrizable spaces $X_{t}$. Since each $X_{t}$ admits a $Z$-embedding into $Y_{t}$, it follows from Proposition \ref{P:zset2} that the product $\prod\{ X_{t} \colon t \in T\}$, and hence $X$, admits a $Z_{\tau}$-embedding into $\prod\{ Y_{t} \colon t \in T\}$.
\end{proof}

%%%%%%%%%%%%%%%%%%%%%%%%%%%%%%%%%%%%%%%%%%%%%%%%%%%%%%%%%%%%%%%%%%%%
%%%%%%%%%%%%%%%%%%%%%%%%%%%%%%%%%%%%%%%%%%%%%%%%%%%%%%%%%%%%%%%%%%%

\subsection{Mapping replacement}\label{SS:mapping}

For a $Z$-set $Z\subset I^{\omega}$ and a closed subset $Y$ of a metrizable compactum $X$, the standard mapping replacement allows us to approximate a map $f \colon X \to I^{\omega}$ by a map $g \colon X \to I^{\omega}$ in such a way that $g|Y = f|Y$ and $g(X\setminus Y) \cap Z = \emptyset$. The key here is that the space $C(X, I^{\omega})$ is completely metrizable and hence possesses Baire property. Below we prove needed versions of mapping replacement for $Z_{\tau}$-sets Tychonov cubes by using spectral technique.

\begin{pro}\label{L:main}
Let $\tau \geq \omega$ and $Z$ be a $Z_{\tau}$-set in $I^{\tau}$. Suppose also that $Y \subset X$ is a closed subset of a compactum $X$. Then for any map $f \colon X \to I^{\tau}$ and any collection $\{ K_{\alpha} \colon 1\leq \alpha < \tau\}$ of compact subsets of $X$ with $\bigcup_{1\leq \alpha < \tau} K_{\alpha} \cap Y = \emptyset$ there exists a map $g \colon X \to I^{\tau}$ such that $g|Y = f|Y$ and $f(\bigcup_{1\leq \alpha < \tau} K_{\alpha}) \subset I^{\tau}\setminus Z$.
\end{pro}
\begin{proof}
Let $|T| = \tau$ and $\{T_{\alpha} \colon \alpha < \tau\}$ be a collection of subsets supplied by Proposition \ref{P:zset2} corresponding to the $Z_{\tau}$-set $Z$.

We are going to construct maps $g_{\alpha} \colon X \to I^{T_{\alpha}}$ as follows. Let $g_{0} = \pi_{T_{0}}\circ f \colon X \to I^{T_{0}}$. Suppose that for each $\beta < \alpha$ we have already constructed $g_{\beta}$ satisfying the following conditions:

\begin{itemize}
  \item[(i)] 
$\pi_{T_{\delta}}^{T_{\beta}} \circ g_{\beta} = g_{\delta}$, whenever $\delta < \beta  < \alpha$; 
  \item[(ii)] 
$g_{\beta} = \lim\{ g_{\delta} \colon \delta < \beta \}$, whenever $\beta< \alpha$ is a limit ordinal;
  \item[(iii)] 
$g_{\beta}(K_{\beta}) \cap \pi_{T_{\beta}}(Z) = \emptyset$, whenever $1 \leq \beta < \alpha$; 
  \item[(iv)]
 $g_{\beta}|Y = \pi_{T_{\beta}}\circ f|Y$, whenever $\beta < \alpha$.
\end{itemize}

First consider the case $\alpha = \beta +1$. Since $\pi_{T_{\alpha}}(Z)$ is a fibered $Z$-set in $I^{T_{\alpha}}$ with respect to the projection $\pi_{T_{\beta}}^{T_{\alpha}} \colon I^{T_{\alpha}} \to I^{T_{\beta}}$ there exists a map $h_{\alpha} \colon I^{T_{\alpha}} \to I^{T_{\alpha}}$ such that $h_{\alpha}(K_{\alpha}) \cap \pi_{T_{\alpha}}(Z) = \emptyset$ and $\pi_{T_{\beta}}^{T_{\alpha}}\circ h_{\alpha} = \pi_{T_{\beta}}^{T_{\alpha}}$. Let $s \colon I^{T_{\beta}} \to I^{T_{\alpha}}$ be a section of the projection $\pi_{T_{\beta}}^{T_{\alpha}} \colon I^{T_{\alpha}} \to I^{T_{\beta}}$. Next consider the map $r_{\alpha} \colon Y \cup K_{\alpha} \to I^{T_{\alpha}}$ which coincides with $\pi_{T_{\alpha}}\circ f$ on $Y$ and with $h_{\alpha} \circ s \circ g_{\alpha}$ on $K_{\alpha}$. Straightforward verification shows that the following diagram of unbroken arrows commutes:

\[
        \xymatrix{
            Y \cup K_{\alpha} \ar^(0.55){r_{\alpha}}[rr] \ar_{\text{incl}}@{_{(}->}[dd] & & I^{T_{\alpha}} \ar^{\pi_{T_{\beta}}^{T_{\alpha}}}[dd]\\
    & & \\
            X \ar^{g_{\beta}}[rr] \ar@{.>}^{g_{\alpha}}[uurr]  & & I^{T_{\beta}}   \\
        }
      \]

\medskip

\noindent Consequently, by softness of the projection $\pi^{T_{\alpha}}_{T_{\beta}} \colon I^{T_{\alpha}} \to I^{T_{\beta}}$, there is a map $g_{\alpha} \colon X \to I^{A}$ (the dotted arrow in the diagram) such that $g_{\alpha}|(Y \cup K_{\alpha}) = r_{\alpha}$ and $\pi_{T_{\beta}}^{T_{\alpha}}\circ g_{\alpha} = g_{\beta}$. It is also clear that $g_{\alpha}(K_{\alpha}) \cap \pi_{T_{\alpha}}(Z) = \emptyset$ and $g_{\alpha}|Y = \pi_{T_{\alpha}}\circ f$.

If $\alpha = \lim\{ \beta \colon \beta < \alpha\}$, then let $g_{\alpha} = \lim\{ g_{\beta} \colon \beta < \alpha\} \colon X \to I^{T_{\alpha}}$. 

This completes inductive construction. Finally, let $g = \lim\{ g_{\alpha} \colon \alpha < \tau\} \colon X \to I^{T}$ be the limit map. It is clear that $g|Y = f|Y$ and $g(K_{\alpha}) \cap Z = \emptyset$ for each $\alpha$.
\end{proof}

\begin{cor}\label{C:1}
Let $\tau > \omega$. For any map $f \colon X \to Y$ between $Z_{\tau}$-sets of the Tychonov cube $I^{\tau}$ there exists a proper map $g \colon I^{\tau}\setminus X \to I^{\tau}\setminus Y$ such that $f = \bar{g}|X$, where $\bar{g} \colon I^{\tau} \to I^{\tau}$ is the extension of $g$.
\end{cor}
\begin{proof}
Let $\bar{f} \colon I^{\tau} \to I^{\tau}$ be an extension of $f$. By Proposition \ref{L:main}, there exists a map $\bar{g} \colon I^{\tau} \to I^{\tau}$ such that $\bar{g}|X = \bar{f}|X = f$ and $\bar{g}(I^{\tau}\setminus X ) \subset I^{\tau}\setminus Y$. Clearly, $g = \bar{g}|(I^{\tau}\setminus X) \colon I^{\tau}\setminus X  \to I^{\tau}\setminus Y$ is a proper map with required properties.
\end{proof}

\begin{lem}\label{L:countable}
Let $B \subset A$ and $|A\setminus B| =\omega$. Suppose that $Z$ is a fibered $Z$-set in $I^{A}$ with respect to the projection $\pi_{B}^{A} \colon I^{A} \to I^{B}$.  Suppose also that $X$ is closed in  $I^{A}$ and we are given an embedding $f \colon X \to Z$ and a map $g \colon I^{B} \to I^{B}$ such that $\pi_{B}^{A}\circ f = g\circ \pi_{B}^{A}|X$. Then there exists a fibered $Z$-embedding $h \colon I^{A} \to I^{A}$ such that $\pi_{B}^{A}\circ h = g\circ \pi_{B}^{A}$, $h|X = f|X$ and $h(I^{A}\setminus X) \subset I^{A}\setminus Z$.
\end{lem}
\begin{proof}
It is possible, using \cite[Theorem 1.3.4]{chibook} for $\omega$-spectra, to find a countable subset $C \subset B$, an embedding $f_{0} \colon \pi_{C\cup (A\setminus B)}^{A}(X) \to \pi_{C \cup (A\setminus B)}^{A}(Z)$ and a map $g_{0} \colon I^{C} \to I^{C}$, satisfying the following conditions:

\begin{itemize}
  \item[(i)]
$\pi_{C \cup (A\setminus B)}^{A}(Z)$ is a fibered $Z$-set in $I^{C \cup (A\setminus B)}$ with respect to the projection $\pi^{C \cup (A\setminus B)}_{C} \colon I^{C\cup (A\setminus B)} \to I^{C}$;
  \item[(ii)] 
$f_{0}\circ \pi_{C\cup (A\setminus B)}^{A}|X = \pi_{C\cup (A\setminus B)}^{A}\circ f$;
  \item[(iii)] 
$g_{0}\circ \pi_{C\cup (A\setminus B)}^{A}= \pi_{C\cup (A\setminus B)}^{A}\circ g$;
  \item[(iv)] 
$\pi_{C}^{C \cup (A\setminus B)}\circ f_{0} = g_{0}\circ \pi_{C}^{C \cup (A\setminus B)}|\pi_{C \cup (A\setminus B)}^{A}(X)$. 
\end{itemize}

Let $\bar{h}\colon I^{C\cup (A\setminus B)} \colon I^{C\cup (A\setminus B)}$ be a map such that $\pi_{C}^{C\cup (A\setminus B)}\circ \bar{h} = g_{0}\circ \pi_{C}^{C\cup (A\setminus B)}$ and $\bar{h}|\pi_{C\cup (A\setminus B)}^{A}(X) = f_{0}$. Next consider the space (in the compact-open topology) $C^{\bar{h}}(I^{C\cup (A\setminus B)}, I^{C\cup (A\setminus B)})$ of all maps $h \colon I^{C\cup (A\setminus B)} \to I^{C\cup (A\setminus B)}$ such that $h\circ \pi_{C}^{C\cup (A\setminus B)} = \bar{h}\circ \pi_{C}^{C\cup (A\setminus B)}$ and $h|\pi^{A}_{C\cup (A\setminus B)}(X) = f_{0}$. It follows from \cite{tw} that the set $S$ of fibered $Z$-embeddings is dense and $G_{\delta}$-subset in $C^{\bar{h}}(I^{C\cup (A\setminus B)}, I^{C\cup (A\setminus B)})$. Moreover, the set $R$ of maps with $h(I^{C \cup A\setminus B)} \setminus \pi_{C\cup (A\setminus B)}^{A}(X)) \cap \pi_{C\cup (A\setminus B)}^{A}(Z) = \emptyset$ is also dense and $G_{\delta}$ in $C^{\bar{h}}(I^{C\cup (A\setminus B)}, I^{C\cup (A\setminus B)})$. Consequently, since $C^{\bar{h}}(I^{C\cup (A\setminus B)}, I^{C\cup (A\setminus B)})$ is completely metrizable, $S \cap R \not= \emptyset$. Take any $h_{0} \in S \cap R$. There is precisely one map $h \colon I^{A}\to I^{A}$ such that $\pi_{C\cup (A\setminus B)}^{A}\circ h = h_{0}\circ \pi_{C\cup (A\setminus B)}^{A}$ and $\pi_{B}^{A}\circ h = \pi_{B}^{A}$. It follows from the construction that $h$ satisfies all required properties. 
\end{proof}

\begin{pro}\label{P:adjusted}
Let $|T| = \tau \geq \omega$ and $Z$ be a $Z_{\tau}$-set in the Tychonov cube $I^{T}$. Suppose that $Y$ is closed in a compactum $X$ of weight $\leq \tau$ and $f \colon X \to I^{T}$ is a map such that $f(Y) \subset Z$ and $f|Y \colon Y \to Z$ is an embedding. Then there exists a $Z_{\tau}$-embedding $h \colon X \to I^{T}$ such that $h|Y = f|Y$ and $h(X\setminus Y) \subset I^{T}\setminus Z$. 
\end{pro}
\begin{proof}
Without loss of generality we may assume that $X = I^{T}$. Using Corollaries \ref{C:spectralwcontrol}, \ref{C:spectralwcontrolh} and Lemma \ref{L:zset1} it is easy to construct subsets $T_{\alpha} \subset T$ and maps $f_{\alpha} \colon I^{T_{\alpha}} \to I^{T_{\alpha}}$, $\alpha < \tau$, satisfying the following properties: 
\begin{itemize}
  \item[(a)]
  $|T_{0}| = \omega$; 
  \item[(b)]
 $T_{\alpha} \subset T_{\alpha +1}$ and $|T_{\alpha +1}\setminus T_{\alpha}| = \omega$;
  \item[(c)]
 $T = \cup\{ T_{\alpha} \colon \alpha < \tau\}$ and $T_{\alpha} = \cup\{ T_{\beta} \colon \beta < \alpha\}$ for each limit ordinal $\alpha < \tau$;
  \item[(d)]
  $\pi_{T_{\alpha +1}}(Z)$ is a fibered $Z$-set in $I^{T_{\alpha +1}}$ with respect to the projection $\pi^{T_{\alpha +1}}_{T_{\alpha}} \colon I^{T_{\alpha +1}} \to I^{T_{\alpha}}$;
  \item[(e)]
 $\pi_{T_{\alpha}}^{T_{\alpha +1}}\circ f_{\alpha +1} = f_{\alpha}\circ \pi_{T_{\alpha}}^{T_{\alpha +1}}$;
  \item[(f)]
  $f = \lim\{ f_{\alpha} \colon \alpha < \tau\}$ and $f_{\alpha} = \lim\{ f_{\beta} \colon \beta < \alpha\}$ for each limit ordinal $\alpha < \tau$; 
  \item[(g)]
 $f_{\alpha}|\pi_{T_{\alpha}}(Y) \colon \pi_{T_{\alpha}}(Y) \to \pi_{T_{\alpha}}(Z)$ is an embedding.
 \end{itemize} 

In order to construct required $Z_{\tau}$-embedding $h \colon I^{T} \to I^{T}$ we proceed by induction. Let $h_{0} = f_{0}$. Supposing that $f_{\beta}$'s have been constructed for all $\beta < \alpha$, construction of $f_{\alpha}$ for non-limit $\alpha$ is straightforward by using Lemma \ref{L:countable}. For a limit $\alpha$, we set $h_{\alpha} = \lim\{ h_{\beta} \colon \beta < \alpha\}$. Finally, required embedding is defined by letting $h = \lim\{ h_{\alpha} \colon \alpha < \tau\}$. Proposition \ref{P:zset2} guarantees that $h$ is a $Z_{\tau}$-embedding. By construction, $h|X = f|X$ and $h(I^{T}\setminus X) \subset I^{T}\setminus Z$.
\end{proof}

%%%%%%%%%%%%%%%%%%%%%%%%%%%%%%%%%%%%%%%%%%%%%%%%%%%%%%%%%%%%
%%%%%%%%%%%%%%%%%%%%%%%%%%%%%%%%%%%%%%%%%%%%%%%%%%%%%%%%

\section{Extension properties of the Stone-\v{C}ech corona}\label{S:corona}

We begin by introducing the following concept (compare to \cite{nov}, \cite{Michael}).

\begin{defin}\label{D:main}
A locally compact space $Y$ is a proper absolute extensor for a locally compact space $X$ (notation: $Y \in \operatorname{AE}_{p}(X)$) if any proper map $f \colon A \to Y$, defined on a closed $C^{\ast}$-embedded subset $A$ of $X$, admits a proper extension $\bar{f} \colon X \to Y$.
\end{defin}

Recall that regular closed subsets are closures of open subsets. 

\begin{thm}\label{T:dim}
Let $X$ be a  locally compact space which can be covered by at most $\tau$ compact subsets and each regular closed subset of which is $C^{\ast}$-embedded. Let also $L$ be a compact  $\operatorname{ANR}$-space embedded into the Tychonov cube $I^{\tau}$ as a $Z_{\tau}$-set. Then the following conditions are equivalent:
\begin{itemize}
  \item[(a)]
$L \in \operatorname{AE}(\beta X \setminus X)$; 
  \item[(b)]
$I^{\tau} \setminus L \in \operatorname{AE}_{p}(X)$. 
\end{itemize}  
\end{thm}
\begin{proof}
(a) $\Longrightarrow$ (b). 
Let $f \colon A \to I^{\tau} \setminus L$ be a proper map defined on a closed $C^{\ast}$-embedded subspace $A \subset X$. Then $\beta A = \operatorname{cl}_{\beta X}A$ and there is an extension $\bar{f} \colon \operatorname{cl}_{\beta X}A \to I^{\tau}$ of $f$. Since $f$ is proper, it follows that $\bar{f}(\operatorname{cl}_{\beta X}A \setminus A) \subset L$. By (a), $\bar{f}|(\operatorname{cl}_{\beta X}A \setminus A) \colon (\operatorname{cl}_{\beta X}A \setminus A) \to L$ can be extended to a map $g \colon \beta X \setminus X \to L$. Since $A \cup (\beta X \setminus X)$ is closed in $\beta X$ there exists a map $G \colon \beta X \to I^{\tau}$ such that $G|(\beta X \setminus X) = g$ and $G|A = f$. 

Using spectral theorem for $\tau$-spectra \cite[Theorem 1.3.4]{chibook}, we can find a compact space $Y$ of weight $\leq \tau$, and maps $p \colon \beta X \to Y$, $q \colon Y \to I^{\tau}$ such that $G = q \circ p$ and $\beta X \setminus X = p^{-1}(p(\beta X \setminus X))$. By Proposition \ref{L:main}, there exists a map $H \colon Y \to I^{\tau}$ such that $H|p(A\cup (\beta X \setminus X)) = q|p(A\cup (\beta X \setminus X))$ and $H(Y \setminus  p(A \cup (\beta X \setminus X))) \subset I^{\tau}\setminus L$.

It only remains to note that the map $F = H \circ p \colon \beta X \to I^{\tau}$ has the following properties: $F|(A \cup (\beta X \setminus X)) = G|(A\cup (\beta X \setminus X))$ and $F(X\setminus A ) \subset I^{\tau}\setminus L$. Consequently, $\widetilde{f} = F|X \colon X \to I^{\tau} \setminus L$ is a proper map extending $f$.

(b) $\Longrightarrow$ (a). Let $f \colon A \to L$ be a map defined on a closed subspace $A \subset \beta X \setminus X$. Since $L$ is an $\operatorname{ANR}$-space, we may assume without loss of generality that $f$ is already defined on the closure  $\operatorname{cl}_{\beta X}U $ of an open neighborhood $U$ of $A$ in $\beta X$. Note that $\operatorname{cl}_{\beta X}U = \operatorname{cl}_{\beta X}(U \cap X) = \operatorname{cl}_{\beta X}(\operatorname{cl}_{X}(U \cap X))$ and that according to our assumption $\operatorname{cl}_{X}(X\cap U)$ is $C^{\ast}$-embedded in $X$.  Since $\kappa (\operatorname{cl}_{X}(U\cap X)) \leq \tau$, we conclude, by Proposition \ref{L:main}, that there esxists a map $g \colon \operatorname{cl}_{\beta X}U \to I^{\tau}$ such that $g|(\operatorname{cl}_{\beta X}U \setminus \operatorname{cl}_{X}(U\cap X))= f|(\operatorname{cl}_{\beta X}U \setminus \operatorname{cl}_{X}(U\cap X))$ and $g(\operatorname{cl}_{X}(U \cap X)) \subset I^{\tau}\setminus L$.  By (b), the proper map $g|\operatorname{cl}_{X}(U \cap X) \colon \operatorname{cl}_{X}(U \cap X) \to I^{\tau}\setminus L$ has a proper extension $G \colon X \to I^{\tau}\setminus L$. Since $G$ is proper, its Stone-\v{C}ech extension $\widetilde{G}\colon \beta X \to I^{\tau}$ sends $\beta X \setminus X$ into $L$. Straightforward verification shows that $\widetilde{F}|A = f$.
\end{proof}

\begin{cor}\label{C:first}
Let $L$ be metrizable $\operatorname{ANR}$-compact space embedded into the Hilbert cube $I^{\omega}$ as a $Z$-set. Then the following conditions are equivalent for any locally compact and  Lindel\"{o}f space $X$:
\begin{itemize}
\item[(a)]
$L \in \operatorname{AE}(\beta X \setminus X)$;
\item[(b)]
$I^{\omega} \setminus L \in \operatorname{AE}_{p}(X)$;
\item[(c)]
$\operatorname{Cone}(L)\setminus L \in \operatorname{AE}_{p}(X)$.
\end{itemize}
\end{cor}
\begin{proof}
Equivalence of (a) and (b) follows from Theorem \ref{T:main} since $\kappa(X) \leq \omega$ for any locally compact and Lindel\"{o}f space $X$. 

To prove the remaining equivalence, first note that by Edwards' theorem \cite[Corollary 2.3.23]{chibook}, $I^{\omega}\times \operatorname{Cone}(L)$ is homeomorphic to the Hilbert cube $I^{\omega}$. Further, by Chapman's complementt theorem \cite{cha}, the complements $I^{\omega}\setminus L$ and $I^{\omega}\times \operatorname{Cone}(L) \setminus I^{\omega}\times L = I^{\omega}\times (\operatorname{Cone}(L)\setminus L)$ are homeomorphic. Finally note that $I^{\omega}\times (\operatorname{Cone}(L)\setminus L) \in \operatorname{AE}_{p}(X)$ precisely when $\operatorname{Cone}(L)\setminus L \in \operatorname{AE}_{p}(X)$.
\end{proof}

\begin{cor}\label{C:dim}
The following conditions are equivalent for any locally compact and Lindel\"{o}f space $X$:
\begin{itemize}
  \item[(a)]
$\dim (\beta X \setminus X) \leq n$; 
  \item[(b)] 
$\dim_{p}X \leq n+1$, i.e. any proper map $f \colon A \to {\mathbb R}^{n+1}$, defined on a closed subspace $A \subset X$, can be extended to a proper map $\bar{f} \colon X \to {\mathbb R}^{n+1}$. 
\end{itemize} 
\end{cor}

%%%%%%%%%%%%%%%%%%%%%%%%%%%%%%%%%%%%%%%%%%%%%%%%%%%%%%%%%%%%%%%%%%%%%%%%%%%%%%%%%%%%%%%%%%%%%%%%

\begin{thm}\label{T:main}
Let $\tau >  \omega$, $L$ be a compact $\operatorname{ANR}$ embedded into $I^{\tau}$ as a $Z_{\tau}$-set and $X$ be a $Z_{\tau}$-set in $I^{\tau}$. Then the following conditions are equivalent:
\begin{itemize}
  \item[(i)]
$X \in \operatorname{AE}([L])$; 
  \item[(ii)]
$I^{\tau}\setminus X \in \operatorname{AE}_{p}^{\tau}([I^{\tau}\setminus L])$.  
\end{itemize}
\end{thm}
\begin{proof}
(ii) $\Longrightarrow$ (i). We need to show that a  map $f \colon A \to X$, defined on a closed subset $A \subset Y$, where $Y$ is a compact space of weight $\leq \tau$ such that $L \in \operatorname{AE}(Y)$, has a continuous extension $g \colon Y \to X$. By Corollary \ref{C:embedding}, we masy assume that $A$ is embedded into $K = I^{\tau}$ as a $Z_{\tau}$-set. Note that $K$ is then the Stone-\v{C}ech compactification of $K\setminus A$ (Lemma \ref{L:SC}). Similarly, we may assume that $Y$ is also embedded into (a different copy of) $I^{\tau}$ as a $Z_{\tau}$-set. By Proposition \ref{P:adjusted}, it is possible to embed $K$ into $I^{\tau}$ in such a way that $K\cap Y = A$. Let $\bar{f} \colon A \to I^{\tau}$ denote an extension of $f$. Using Proposition \ref{L:main} we can find a map $h \colon K \to I^{\tau}$ such that $h|A = f$ and $h(K\setminus A) \subset I^{\tau}\setminus X$. Note that $h|(K\setminus A) \colon K\setminus A \to I^{\tau}\setminus X$ is proper. Since $K$ is the Stone-\v{C}ech compactification of $K\setminus A$, it follows that $K\setminus A$ is $C^{\ast}$-embedded in $I^{\tau}\setminus Y$. Next note that, by Lemma \ref{L:SC}(iii), $I^{\tau}\setminus Y$ satisfies assumptions of Theorem \ref{T:dim}. Consequently, since $L \in \operatorname{AE}(Y)$, we conclude by  Theorem \ref{T:dim}  that $I^{\tau} \setminus L \in \operatorname{AE}(I^{\tau}\setminus Y)$. Then, by(ii), the proper map $h|(K\setminus A) \colon K \setminus A \to I^{\tau}\setminus X$ admits a proper extension $\bar{h} \colon I^{\tau} \setminus Y \to I^{\tau}\setminus X$. Let $\bar{g} \colon I^{\tau} \to I^{\tau}$ be the extension of $\bar{h}$. Properness of $\bar{h}$ implies that $\bar{g}(Y) \subset X$. Then $g = \bar{g}|Y \colon Y \to X$ is the required extension of $f$.

(i) $\Longrightarrow$ (ii). Let now $f \colon B \to  I^{\tau}\setminus X$ be a proper map defined on a closed and $C^{\ast}$-embedded subset $B$ of a locally compact space $Y$ of weight $\leq \tau$ such that $I^{\tau}\setminus L \in \operatorname{AE}_{p}(Y)$. We need to construct a proper extension $\bar{f} \colon Y \to I^{\tau}\setminus X$ of $f$. Since $B$ is $C^{\ast}$-embedded in $Y$ it follows that $\operatorname{cl}_{\beta Y}B$ is the Stone-\v{C}ech compactification of $B$. Consequently there is the extension $g \colon \operatorname{cl}_{\beta Y}B \to I^{\tau}$ of $f$. Properness of $f$ implies that $g(\operatorname{cl}_{\beta Y}B \setminus B) \subset X$. Since $I^{\tau}\setminus L \in \operatorname{AE}_{p}(Y)$ we conclude, by Theorem \ref{T:dim}, that $L \in \operatorname{AE}(\beta Y \setminus Y)$. Thus, by (i), there exists an extension $h \colon \beta Y \setminus Y \to X$ of $g|(\operatorname{cl}_{\beta Y}B \setminus B) \colon \operatorname{cl}_{\beta Y}B \setminus B \to X$. Next consider the closed subset $A = (\beta Y \setminus Y) \cup B$ of $\beta Y$ and the map $h^{\prime} \colon A \to I^{\tau}$ defined be letting

\[ h^{\prime}(y) = 
\begin{cases}
h(y),\;\text{if}\; y \in \beta Y \setminus Y;\\
f(y), \;\text{if}\; y \in B \; .\\
\end{cases}
\] 

Next consider any extension $\bar{h} \colon \beta Y \to I^{\tau}$ of $h^{\prime}$. By Proposition \ref{L:main}, we can find a map $\bar{g} \colon \beta Y \to I^{\tau}$ such that $\bar{g}|A = h^{\prime}$ and $\bar{g}(\beta Y \setminus A) \subset I^{\tau}\setminus X$. Straighforward verification shows that $\bar{g}(Y) \subset I^{\tau}\setminus X$ and consequently $\bar{f} = \bar{g}|Y \colon Y \to I^{\tau}\setminus X$ is proper. It only remains to note that $\bar{f}|B = f$. 
\end{proof}

%%%%%%%%%%%%%%%%%%%%%%%%%%%%%%%%%%%%%%%%%%%
%%%%%%%%%%%%%%%%%%%%%%%%%%%%%%%%%%%%%%%%%%%%%

\section{Categorical Isomoprhisms}\label{S:category}

Recall (Corollary \ref{C:1}) that for any map $f \colon X \to Y$ between $Z_{\tau}$-sets of the Tychonov cube we can find a proper map $g \colon I^{\tau}\setminus X \to I^{\tau}\setminus Y$ such that $f = \widetilde{g}|X$, where $\widetilde{g} \colon I^{\tau} \to I^{\tau}$ is the unique extension of $g$. If $f$ is a homeomorphism, we may assume that $g$ is also a homeomorphism (this a $Z_{\tau}$-set unknotting theorem, \cite[Theorem 8.5.4]{chibook}). In other words, any map between $Z_{\tau}$-sets of the Tychonov cube can be obtained as the restriction of the Stone-\v{C}ech extension of a proper map between their complements, i.e. the correspondence $\lambda \colon C_{p}(I^{\tau}\setminus X, I^{\tau}\setminus Y) \to C(X,Y)$, defined by $\lambda(g) = \widetilde{g}|X$  is surjective. Below we show that (up to certain equivalence relation) $\lambda$ is in fact a bijection. Here we extend considerations of \cite{cdkm}, carried out for the Hilbert cube, to the Tychonov cube.

Let ${\mathcal Z}_{\tau}$ denote the category of $Z_{\tau}$-sets of the cube $I^{\tau}$ and their continuous maps. Let also ${\mathcal C}_{p}(\mathcal {Z}_{\tau})$ denote the category whose objects are complementts of $Z_{\tau}$-sets in $I^{\tau}$ and whose morphisms are equivalence classes of proper maps with respect to the following relation: two proper maps are equivalent if they are close in the continuously controlled (by the compactification $I^{\tau}$) coarse structure (\cite[Remark 2.29]{roe}). Recall that two proper maps $g_{1}, g_{2} \colon I^{\tau}\setminus X \to I^{\tau}\setminus Y$ are close if $\widetilde{g}_{1}(x) = \widetilde{g}_{2}(x)$ for any $x \in X$. The equivalence class with representative $g$ will be denoted by $\{ g\}$. With this in mind we have

\begin{pro}\label{P:category}
Let $\tau > \omega$. Then the correspondence $\lambda \colon {\mathcal C}_{p}({\mathcal Z}_{\tau}) \to {\mathcal Z}_{\tau}$, defined by letting:
\begin{itemize}
\item[(i)]
For $I^{\tau}\setminus X \in \mathcal{O}\mathcal{B}({\mathcal C}_{p}({\mathcal Z}_{\tau}))$, $\lambda (I^{\tau}\setminus X) = X$,
\item[(ii)]
For $\{ g\} \colon I^{\tau}\setminus X \to I^{\tau}\setminus Y \in \mathcal{MOR}({\mathcal C}_{p}(\mathcal{Z}_{\tau}))$, $\lambda(\{ g\} ) = \widetilde{g}|X$,
\end{itemize}
\noindent is an isomorphism of categories.
\end{pro}
\begin{proof}
Structurally the proof follows that of \cite[Theorem 2]{cdkm}, but is much simpler and is left to the reader. Fact that $\lambda$ is well defined on morphisms is a direct consequence of the definition of the closeness relation. The fact that $\lambda$ is surjective on morphisms, as noted above, follows from Corollary \ref{C:1}.
\end{proof}

Next we consider homotopy categories. Let $\mathcal{H}({\mathcal Z}_{\tau})$ denote the category whose objects are same as in $\mathcal{Z}_{\tau}$ and morphisms are homotopy classes of maps. Similarly $\mathcal{H}_{p}(\mathcal{Z}_{\tau})$ denotes the category whose objects are same as in $\mathcal{C}_{p}(\mathcal{Z}_{\tau})$ and morphisms are proper homotopy classes of proper maps. First, we need the following observation.

\begin{lem}\label{L:2}
Let $\tau > \omega$ and $f_{0}, f_{1} \colon X \to Y$ be two maps between $Z_{\tau}$-sets of the Tychonov cube $I^{\tau}$. Suppose also that $g_{0}, \colon g_{1} \colon I^{\tau}\setminus X \to I^{\tau}\setminus Y$ are proper maps such that $f_{k} = \widetilde{g}_{k}|X$, $k = 0,1$. Then $f_{0}\simeq f_{0}$ iff $g_{0} \simeq_{p} g_{1}$.
\end{lem}
\begin{proof}
Let $F \colon X \times [0,1] \to Y$ be a homotopy between $f_{0}$ and $f_{1}$. Consider the map $H \colon X\times [0,1] \cup I^{\tau} \times \{0,1\} \to I^{\tau}$ defined by letting

\[ H(z,t) = 
\begin{cases}
F(z,t), \; \text{if}\; (z,t) \in X\times [0,1],\\
\widetilde{g}_{k}(z),\; \text{if}\; z \in I^{\tau} \times \{0,1\} , k = 0,1.\\
\end{cases}
\]

By Proposition \ref{L:main}, there exists a map $\widetilde{G} \colon I^{\tau}\times [0,1] \to I^{\tau}$ such that $\widetilde{G}|(X\times [0,1] \cup I^{\tau} \times \{0,1\}) = H$ and $\widetilde{G}(I^{\tau}\times [0,1] \setminus (X\times [0,1] \cup I^{\tau} \times \{0,1\})) \subset I^{\tau}\setminus Y$. Clearly, $G = \widetilde{G}|((I^{\tau}\setminus X) \times [0,1]) \colon (I^{\tau}\setminus X) \times [0,1] \to I^{\tau}\setminus Y$ is a proper homotopy between $g_{0}$ and $g_{1}$.

Conversely, suppose that $G \colon (I^{\tau}\setminus X) \times [0,1] \to I^{\tau}\setminus Y$ is a proper homotopy between $g_{0}, g_{1} \colon I^{\tau}\setminus Y \to I^{\tau}\setminus X$. Note that by Lemma \ref{L:SC}, $I^{\tau}\times [0,1]$ is the Stone-\v{C}ech compactification of the product $(I^{\tau}\setminus X) \times [0,1]$. Consequently, $G$ admits an extension $\widetilde{G} \colon I^{\tau} \times [0,1]$ such that $\widetilde{G}(X \times [0,1]) \subset Y$. It is clear that $H = \widetilde{G}|X\times [0,1] \colon X \times [0,1]\to Y$ is a homotopy between $f_{0}$ and $f_{1}$.
\end{proof}

\begin{cor}\label{L:close}
Let $\tau > \omega$ and $X$ and $Y$ be $Z_{\tau}$-sets in $I^{\tau}$. If proper maps $g_{0}, g_{1} \colon T^{\tau}\setminus X \to I^{\tau}\setminus Y$ are close with respect to the continuously controlled coarse structure induced by $I^{\tau}$, then $g_{0}$ and $g_{1}$ are properly homotopic.
\end{cor}
\begin{proof}
\cite[Theorem 2.27]{roe} implies that coarsely close proper maps coincide on the Stone-\v{C}ech corona. Consequently, by Lemma \ref{L:2}, they are properly homotopic.
\end{proof}

Now let us define a functor $\mu  \colon {\mathcal H}_{p}({\mathcal Z}_{\tau}) \to {\mathcal H}({\mathcal Z}_{\tau})$ between these homotopy categories. The following statement is parallel to \cite[Proposition 10]{cdkm}.

\begin{pro}\label{P:hcategory}
Let $\tau > \omega$. Then the correspondence $\mu \colon {\mathcal H}_{p}({\mathcal Z}_{\tau}) \to {\mathcal H}({\mathcal Z}_{\tau})$, defined by letting:
\begin{itemize}
\item[(i)]
For $I^{\tau}\setminus X \in \mathcal{O}\mathcal{B}({\mathcal H}_{p}({\mathcal Z}_{\tau}))$, $\mu (I^{\tau}\setminus X) = X $,
\item[(ii)]
For $[g] \colon I^{\tau}\setminus X \to I^{\tau}\setminus Y \in \mathcal{MOR}({\mathcal H}_{p}(\mathcal{Z}_{\tau}))$, $\mu([g]) = [\widetilde{g}|X]$,
\end{itemize}
\noindent is an isomorphism of categories.
\end{pro}
\begin{proof}
One part of Lemma \ref{L:2} shows that $\mu$ is well defined. The other part guarantees that $\mu$ is surjective on morphisms. The rest is straightforward and left to the reader. 
\end{proof}

%%%%%%%%%%%%%%%%%%%%%%%%%%%%%%%%
%%%%%%%%%%%%%%%%%%%%%%%%%%%%%%%%%%

In light of above considerations and the role the closeness relation associated to the continuously controlled coarse structure (induced by the Stone-\v{C}ech compactification), we would like to investigate this concept a bit further. For locally compact and paracompact spaces \cite[Theorem 2.27]{roe} characterizes close proper maps between such spaces as those whose extensions to the Stone-\v{C}ech compactifications coincide on the Stone-\v{C}ech corona. For Lindel\"{o}f spaces we have the following statement.

\begin{pro}
Let  $f, g \colon X \to Y$ be proper maps between locally compact and Lindel\"{o}f spaces. Then the following conditions are equivalent:
\begin{itemize}
\item[(i)]
$f$ and $g$ are close in the continuously controlled (by the Stone-\v{C}ech compactification) coarse structure;
\item[(ii)]
$\widetilde{f}(x) = \widetilde{g}(x)$ for any $x \in \beta X \setminus X$;
\item[(iii)]
There is a compact subset $C \subset X$ such that $f(x) = g(x)$ for any $x \in X \setminus C$.
\end{itemize}
\end{pro}
\begin{proof}
As mentioned above, \cite[Theorem 2.27]{roe} implies the equivalence (i) $\Longleftrightarrow$ (ii). Let us prove the implication (ii) $\Longrightarrow$ (iii). 
Take a point  $x \in \beta X \setminus X$. Since $X$ is Lindel\"{o}f (actually realcompactness of $X$ suffices here), we can find a sequence $\{ U_{n} \colon n \in \omega\}$ of open neighborhoods of $x$ in $\beta X$ with the following properties:
\begin{itemize}
\item[(1)]
$\operatorname{cl}_{\beta X}U_{n+1} \subset U_{n}$;
\item[(2)]
$\cap_{n\in \omega}U_{n} \subset \beta X \setminus X$.
\end{itemize}

We need the following

{\it Claim. There exists $i \in \omega$ such that $\widetilde{f}|_{U_{i}} = \widetilde{g}|_{U_{i}}$.}

To prove the claim, assume the contrary. By (ii), our assumption implies that $f|_{U_{i}\cap X} \neq g|_{U_{i}\cap X}$ for each $i \in \omega$. Assume that for each $k < n$ we have found $z_{k} \in U_{k}\cap X$ such that
\begin{itemize}
\item[(*)]
$$\{ f(z_{i})\colon i \leq k \} \cap \{ g(z_{i}) \colon i \leq k\} = \emptyset .$$
\end{itemize}
Next, let us construct a desired $z_{n}$. First, for each $i \geq n$, fix $a_{i} \in U_{i} \cap X$ such that
\begin{itemize}
\item[(a)]
$f(a_{i}) \neq g(a_{i})$.
\end{itemize}
Such $a_{i}$'s exist because we assumed that $f|_{U_{i}\cap X} \neq g|_{U_{i}\cap X}$ for each $i \in \omega$. Due to $(2)$ and $(3)$, no infinite subset of $\{ a_{i} \colon i \geq n\}$ is compact. Since $f$ is proper, there exits $n_{1}$ such that
\begin{itemize}
\item[(b)]
$f(a_{i}) \notin \{ g(z_{k}) \colon k \leq n \}$ for each $i > n_{1}$.
\end{itemize}
Similarly, there exists $n_{2}$ such that
\begin{itemize}
\item[(c)]
$g(a_{i}) \notin \{ f(z_{k}) \colon k \leq n \}$ for each $i > n_{2}$.
\end{itemize}
Pick any $i > \operatorname{max}\{ n_{1},n_{2}\}$ and let $z_{n} = a_{i}$. Since $i \geq n$, $z_{n} \in U_{n}$. By (a)--(c), the formula (*) holds for $k = n$. Our construction is complete.

Let $Z = \{ z_{n} \colon n \in \omega\}$. Clearly $Z$ is closed in $X$. Since $f$ and $g$ are closed maps, $f(Z)$ and $g(Z)$ are closed in $Y$. By (*), they are disjoint. Since $Y$ is normal, $\operatorname{cl}_{\beta Y}f(Z)$ and $\operatorname{cl}_{\beta Y}g(Z)$ are also disjoint. Therefore $\widetilde{f}(x) \neq \widetilde{g}(x)$ contradicting the hypotheseis of the lemma. The claim is proved.

By claim, for each $x \in \beta X \setminus X$, we can select an open neighborhood $U_{x}$ of $x$ in $\beta X$ such that $\widetilde{f}|_{U_{x}} = \widetilde{g}|_{U_{x}}$. Let $U = \cup_{x \in \beta X \setminus X}U_{x}$. Then $\widetilde{f}|_{U} = \widetilde{g}|_{U}$. The set $C = \beta X \setminus U$ is a compact subset of $X$ and $f|_{X\setminus C} = g|_{X\setminus C}$.

Implication (iii) $\Longrightarrow$ (ii) is trivial and valid for any spaces. Indeed, let $C$ be a compact subset of $X$ such that $f|(X\setminus C) = g|(X\setminus C)$. Consider a point $x \in \beta X \setminus X$. Since $C$ is closed in $\beta X$ and $x \notin C$, we can find an open neighborhood $U$ of $x$ in $\beta X$ such that $\operatorname{cl}_{\beta X}U \cap C = \emptyset$. The functions $\widetilde{f}|\operatorname{cl}_{\beta X}U$ and  $\widetilde{g}|\operatorname{cl}_{\beta X}U$ coincide on $U \cap X$. Since $U \cap X$ is dense in $\operatorname{cl}_{\beta X}U$, we conclude that $\widetilde{f}|\operatorname{cl}_{\beta X}U = \widetilde{g}|\operatorname{cl}_{\beta X}U$. Consequently, $\widetilde{f}(x) = \widetilde{g}(x)$.
\end{proof}

%%%%%%%%%%%%%%%%%%%%%%%%%%%%%%%%%%%%
%%%%%%%%%%%%%%%%%%%%%%%%%%%%%%%%%%%%%%

\section{Proper Absolute Extensors}\label{S:PAE}

We begin by a local version Definition  \ref{D:main}.

\begin{defin}\label{D:mainl}
A locally compact space $X$ is a proper absolute neighborhood extensor for a locally compact space $Y$ (notation: $X \in \operatorname{ANE}_{p}(Y)$) if any proper map $f \colon A \to X$, defined on a closed subset $A$ of $Y$, admits a proper extension $\bar{f} \colon G \to X$, where $G$ is a closed neighborhood of $A$ in $Y$.
\end{defin}

Below let $\mathcal{LCL}$ denote the class of locally compact and Lindel\"{o}f spaces.

\begin{defin}\label{D:m}
A space $X \in \mathcal{LCL}$ is a proper absolute (neighborhood) extensor (notation: $X \in \operatorname{A(N)E}_{p}$) if $X \in \operatorname{A(N)E}_{p}(Y)$ for any $Y \in \mathcal{LCL}$.
\end{defin} 

\begin{pro}\label{P:retract}
Every proper absolute (neighborhood) extensor is an absolute (neighborhood) extensor.
\end{pro}
\begin{proof}
Let $X$ be a proper absolute neighborhood extensor. Since $X$ is locally compact and Lindel\"{o}f, there exists a proper map $p \colon X \to Y$, where $Y$ is a locally compact space with countable base. We may assume that $Y$ is a closed subspace of $[0,1)\times I^{\omega}$. Let also $i \colon X \to I^{\tau}$ denote an embedding of $X$ into the Tychonov cube $I^{\tau}$, where $\tau = w(X) \geq \omega$. Then the diagonal product $q = p \triangle i \colon X \to [0,1)\times I^{\omega}\times I^{\tau} \approx [0,1)\times I^{\tau}$ is an embedding with $q(X)$ closed in $[0,1) \times I^{\tau}$. We will identify $X$ with $q(X) \subset [0,1) \times I^{\tau}$. Since $X$ is a proper absolute neighborhood extensor, there exist a functionally open neighborhood $G$ of $X$ in $[0,1) \times I^{\tau}$ and a proper retraction $r \colon \operatorname{cl}_{[0,1)\times I^{\tau}}G \to X$. Since, by \cite[Proposition 6.1.4, Lemma 7.1.3]{chibook}, $G$ is an absolute neighborhood extensor, it follows that $X$ too is an absolute neighborhood extensor. 
\end{proof}

As noted in the Introduction, $R^{n}_{+}$ is a proper absolute extensor. The next statement makes this observation formal.

\begin{lem}\label{L:manifold}
Let $M$ be a compact metrizable $\operatorname{A(N)E}$-space and $N$ be a $Z$-set in $M$. If $N$ is also an $\operatorname{A(N)E}$-compactum, then $M\setminus N \in \operatorname{A(N)E}_{p}$.
\end{lem}
\begin{proof}
We only prove the parenthetical part since the absolute case is simpler. Let $f \colon A \to M\setminus N$ be a proper map, defined on a closed subset of a locally compact space $X$. Without loss of generality we may assume that $A$ is functionally closed in $X$. Consider the Stone-\v{C}ech extension $\widetilde{f} \colon \operatorname{cl}_{\beta X}A \to M$ of $f$. Since $f$ is proper, it follows that $\widetilde{f}(\operatorname{cl}_{\beta X}A \setminus A) \subset N$. Since $N$ is an $\operatorname{ANE}$-compactum, the map $\widetilde{f}|(\operatorname{cl}_{\beta X}A \setminus A) \colon \operatorname{cl}_{\beta X}A \setminus A \to N$ can be extended to a map $g \colon G \to N$, defined on an open neighborhood $G$ of $\operatorname{cl}_{\beta X}A \setminus A$ in $\beta X \setminus X$. Since $(\beta X \setminus X)\setminus G$ and $\operatorname{cl}_{\beta X}A$ are disjoint closed subsets of $\beta X$, we can find an open neighborhood $U$ of $\operatorname{cl}_{\beta X}A$ in $\beta X$ such that $\operatorname{cl}_{\beta X}U \cap (\beta X \setminus X) \subset G$. Next consider a map $h \colon (\operatorname{cl}_{\beta X}U \setminus X) \cup \operatorname{cl}_{\beta X}A \to M$, defined by letting

\[ 
h(x) = 
\begin{cases}
g(x), \text{if}\; x \in \operatorname{cl}_{\beta X}U \setminus X \\
f(x), \text{if}\; x \in A.\\
\end{cases}
\]

\noindent Note that $h$ is well defined since $\widetilde{f}$ and $g$ coincide on $\operatorname{cl}_{\beta X}A \setminus A$. Since $M$ is an $\operatorname{ANE}$-compactum, we can extend $h$ to a map $\widetilde{h} \colon \operatorname{cl}_{\beta X}V \to M$, where $V$ is an open subset of $\beta X$ such that $(\operatorname{cl}_{\beta X}U \setminus X) \cup \operatorname{cl}_{\beta X}A \subset V$ and $\operatorname{cl}_{\beta X}V \subset U$. Next choose a function $\alpha \colon \operatorname{cl}_{\beta X}V \to [0,1]$ such that $\alpha^{-1}(0) = (\operatorname{cl}_{\beta X}V \setminus V) \cup A$. This is possible since the Stone-\v{C}ech corona $\beta X \setminus X$ (and consequently $\operatorname{cl}_{\beta X}V \setminus \operatorname{cl}_{X}V$) is functionally closed in $\beta X$ (respectively, in $\operatorname{cl}_{\beta X}V$). Also, since by our assumption, $N$ is a $Z$-set in $M$, there is a homotopy $H \colon M \times [0,1] \to M$ such that $H(m,0) = m$, for any $m \in M$ and $H(m, t) \in M\setminus N$ for any $(m,t) \in M \times (0,1]$. Finally consider a map $f^{\prime} \colon \operatorname{cl}_{\beta X}V \to M$, defined as follows: $f^{\prime}(x) = H(\widetilde{h}(x),\alpha(x))$, $x \in \operatorname{cl}_{\beta X}V$. Note that $f^{\prime}(\operatorname{cl}_{X}V) \subset M\setminus N$ and $f^{\prime}(\operatorname{cl}_{\beta X}V \setminus \operatorname{cl}_{X}V) \subset N$. Consequently, $f^{\prime}| \operatorname{cl}_{X}V \colon  \operatorname{cl}_{X}V \to M\setminus N$ is proper. It only remains to observe that, by construction, $f^{\prime}|A = f$.
\end{proof}

\begin{lem}\label{L:compactification}
Let $X$ be a $\operatorname{AE}_{p}$-space with countable base. Then its one-point compactification $\alpha X = X \cup \{ \infty\}$ is an $\operatorname{AE}$-compactum and the point $\{\infty\}$ is a $Z$-set in $\alpha X$.
\end{lem}
\begin{proof}
Embed $\alpha X$ into the Hilbert cube $I^{\omega}$. Since $X$ is a proper absolute neighborhood extensor, there exists a proper retraction $r \colon I^{\omega} \setminus \{\infty\} \to X$. 
Properness of $r$ guarantees that $r$ has an extension $\widetilde{r} \colon I^{\omega} \to \alpha X$ such that $\widetilde{r}|(I^{\omega}\setminus\{\infty\}) = r$ and $\widetilde{r}(\{\infty\}) = \{\infty\}$. Since $\widetilde{r}$ is also a retraction, it follows that $\alpha X$ is an absolute extensor. Since $\widetilde{r}^{-1}(\{\infty\}) = \{\infty\}$ and $\{\infty\}$ is a $Z$-set in $I^{\omega}$, we conclude that $\{\infty\}$ is a $Z$-set in $\alpha X$ as well.
\end{proof}

\begin{cor}\label{C:iff}
Let $X$ be a locally compact space with countable base. Then the following conditions are equivalent:
\begin{itemize}
\item[(i)]
$X$ is a proper absolute extensor;
\item[(ii)]
$X$ is a proper retract of $[0,1)\times I^{\omega}$;
\item[(iii)]
The one-point compactification $\alpha X = X \cup \{\infty\}$ of $X$ is an absolute extensor in which the point $\{\infty\}$ is a $Z$-set.
\item[(iv)]
There exists a metrizable compactrification $\widetilde{X}$ of $X$ such that $\widetilde{X}$ and $\widetilde{X}\setminus X$ are absolute extensors and the corona $\widetilde{X}\setminus X$ is a $Z$-set in $\widetilde{X}$.
\end{itemize}
\end{cor}
\begin{proof}
(i) $\Longrightarrow$ (ii). Any locally compact space with countable base, in particular, $X$, admits a closed embedding into $[0,1) \times I^{\omega}$. By (i), the identity map $\operatorname{id}_{X}$ has a proper extension $r \colon [0,1)\times I^{\omega} \to X$, which obviously is a retraction. Implication (ii) $\Longrightarrow$ (iii) follows from Lemma \ref{L:compactification}, since $[0,1)\times I^{\omega}$ (and hence $X$ as its proper retract) is a proper absolute extensor. Implication (iii) $\Longrightarrow$ (iv) is trivial and implication (iv) $\Longrightarrow$ (i) follows from Lemma \ref{L:manifold}.
\end{proof}

\begin{pro}\label{P:1}
Let $X$ be a proper absolute extensor of countable weight. Then the following conditions are equivalent:
\begin{itemize}
\item[(i)]
$X$ satisfies $DD^{n}P$ for each $n$;
\item[(ii)]
$X$ is homeomorphic to $[0,1) \times I^{\omega}$.
\end{itemize}
\end{pro}
\begin{proof}
(i) $\Longrightarrow$ (ii). By Corollary \ref{C:iff}(iii),  the one-point compactification $\alpha X = X \cup \{\infty\}$ of $X$ is an absolute extensor in which the point $\{\infty\}$ is a $Z$-set. Then, by (i), $\alpha X$ has the $DD^{n}P$ for each $n$ and by Torun\'{n}czyk's theorem \cite{torunQ}, $\alpha X \approx I^{\omega}$. Therefore $X \approx I^{\omega}\setminus \{\infty\}$. Finally note that by Chapman's Complement Theorem, $I^{\omega} \setminus \{\infty\} \approx [0,1)\times I^{\omega}$.

(ii) $\Longrightarrow$ (i). Trivial.
\end{proof}

\begin{cor}\label{C:product}
If $X$ is a proper absolute extensor of countable weight, then $X \times I^{\omega} \approx [0,1) \times I^{\omega}$.
\end{cor}
\begin{proof}
Note that $X\times I^{\omega}$ is a proper absolute extensor satisfying $DD^{n}P$ for each $n$ and apply Proposition \ref{P:1}
\end{proof}

\begin{thm}\label{T:1}
A proper absolute extensor of weight $\tau > \omega$ is homeomorphic to the product $[0,1)\times I^{\tau}$ if and only if it has the same pseudocharacter at each point.  
\end{thm}
\begin{proof}
Obviously pseudocharacter of each point of the product $[0,1) \times I^{\tau}$ equals to $\tau$. Let now $X$ be a proper absolute retract of weight $\tau$. As in the proof of Proposition \ref{P:retract} we may assume that $X$ is closed in $[0,1) \times I^{A}$, where $|A| = \tau$. Since $X$ is a proper absolute extensor, there exists a proper retraction $r \colon [0,1) \times I^{A} \to X$. Proceeding as in the proof of \cite[Theorem 7.2.8]{chibook}, we can construct a continuous well ordered inverse spectrum ${\mathcal S} = \{ X_{\alpha}, p_{\alpha}^{\alpha +1}, \tau\}$ of length $\tau$, satisfying the following conditions:
\begin{itemize}
\item[(i)]
$X = \lim{\mathcal S}$;
\item[(ii)]
all spaces $X_{\alpha}$ are locally compact and Lindel\"{o}f proper absolute extensors;
\item[(iii)]
all short projections $p_{\alpha}^{\alpha +1} \colon X_{\alpha +1} \to X_{\alpha}$ are trivial bundles with fiber $I^{\omega}$;
\item[(iv)]
the space $X_{0}$ is a locally compact space of countable weight.
\end{itemize}

Then the space $X$ is homeomorphic to the product $X_{0}\times I^{\tau}$. By (ii), (iv) and Corollary \ref{C:product}, $X_{0}\times I^{\omega} \approx [0,1)\times I^{\omega}$ and consequently, $X \approx [0,1) \times I^{\tau}$. 
\end{proof}

%%%%%%%%%%%%%%%%%%%%%%%%%%%%%%%%%%%%%%%%%%%%%%%%%%%%
%%%%%%%%%%%%%%%%%%%%%%%%%%%%%%%%%%%%%%%%%%%%%%%%

%%%%%%%%%%%%%%%%%%%%%%%%%%%%%%%%%%%%%%%%%%%%%%%%%%%%%%
%%%%%%%%%%%%%%%%%%%%%%%%%%%%%%%%%%%%%%%%%%%%%%%%%%%%%%%

\end{document}